\documentclass[10pt]{article}
\usepackage{amssymb}
\usepackage{amsmath,mathtools}
\usepackage{amsfonts}
\usepackage{amsthm}
\usepackage{dsfont}
\usepackage{amssymb}
\usepackage{mathpazo}
\usepackage{fullpage}
\usepackage{enumerate}
\usepackage[colorlinks=true,linkcolor=blue,citecolor=blue,urlcolor=blue]{hyperref}

\newcommand{\Tr}{\operatorname{Tr}}
\newcommand{\I}{\mathds{1}}
\newcommand{\diag}{\operatorname{diag}}

\newcommand{\CC}{\mathbb{C}}
\newcommand{\RR}{\mathbb{R}}

\newcommand{\calA}{\mathcal{A}}

\renewcommand{\H}{\operatorname{H}}
\renewcommand{\P}{\operatorname{P}}

\newcommand\im{\operatorname{im}}

\usepackage{braket}

\theoremstyle{plain}
\newtheorem{theorem}{Theorem}

\newtheorem{proposition}[theorem]{Proposition}

\theoremstyle{definition}

\usepackage{authblk}
\begin{document}

\title{On directional derivatives of trace functionals of the form $A\mapsto\Tr(Pf(A))$}

\author[1]{Mark W.\ Girard}
\date{\today}
\affil[1]{Institute for Quantum Computing at the University of Waterloo, Waterloo, ON, Canada}
\maketitle

\begin{abstract}
 Given a function $f:(0,\infty)\rightarrow\RR$ and a positive semidefinite $n\times n$ matrix $P$, one may define a trace functional on positive definite $n\times n$ matrices as $A\mapsto \Tr(Pf(A))$. For differentiable functions $f$, the function $A\mapsto \Tr(Pf(A))$ is differentiable at all positive definite matrices $A$. Under certain continuity conditions on~$f$, this function may be extended to certain non-positive-definite matrices $A$, and the \emph{directional} derivatives of $\Tr(Pf(A)$ may be computed there. This note presents conditions for these directional derivatives to exist and computes them. These conditions hold for the function $f(x)=\log(x)$ and for the functions $f_p(x)=x^p$ for all $p>-1$. The derivatives of the corresponding trace functionals are computed here.
\end{abstract}

\section{Introduction}
\label{sec:intro}

Let $\H_n$ denote the set of $n\times n$ Hermitian matrices over $\CC$, let $\P_n$ denote the subset of positive semidefinite $n\times n$ matrices, and let $\P_n^+$ denote the positive definite ones. Any function of real numbers $f:(0,\infty)\rightarrow\RR$ can be extended to positive definite matrices by means of the spectral decomposition. Given a positive matrix $A\in\P_n^+$ with spectral decomposition
\[
 A = \sum_{i=1}^n \alpha_i \,v_iv_i^*,
\]
where $\alpha_1,\dots,\alpha_n\in(0,\infty)$ are the eigenvalues of $A$ and $v_1,\dots,v_n\in\CC^n$ are the corresponding normalized eigenvectors, one defines $f(A)$ as
\[
 f(A) = \sum_{i=1}^n f(\alpha_i) \,v_iv_i^*,
\]
where $v^*$ denotes the conjugate transpose of a vector $v\in \CC^n$. For a positive semidefinite $n\times n$ matrix $P\in P_n$ and a function $f:(0,\infty)\rightarrow\RR$, one may define a function $f_P:\P_n^+\rightarrow\RR$ defined as
\begin{equation}\label{fPAdef}
 f_P(A) = \Tr(Pf(A))
\end{equation}
for all positive definite matrices $A\in\P_n^+$. Functions of this type arise frequently, for example, in the study of quantum information theory \cite{Girard2014}. In this note, we investigate continuity and differentiability properties of functionals of the form in \eqref{fPAdef}.

If the function $f:(0,\infty)\rightarrow\RR$ can be continuously extended to be defined at $0$, the function $f_P$ can be continuously extended to be defined at all non-positive-definite $n\times n$ matrices in the natural way. However, if the limit  $\lim_{t\rightarrow0^+}f(t)$ does not exist, it is still possible to define $f_P(A)$ for certain non-positive-definite $n\times n$ matrices~$A$ by restricting $f$ to the subspace spanned by the eigenvectors corresponding to nonzero eigenvalues of $A$. For example, the \emph{quantum relative entropy} of two positive $n\times n$ matrices $P,Q\in\P_n$ is defined by \cite{Watrous2018}
\[
 S(P\lVert Q) = \left\{ \begin{array}{ll}
                         \Tr(P\log P) - \Tr(P\log Q) & \text{whenever } \im(P)\subseteq\im(Q)\\
                         +\infty & \text{otherwise},
                        \end{array}
 \right.
\]
where $\im(P)$ denotes the image of $P$ and $\Tr(P\log Q)$ has a natural interpretation whenever $\im(P)\subseteq\im(Q)$. Indeed, for any continuous function $f:(0,\infty)\rightarrow\RR$ and any matrix $P\in\P_n$, it is natural to define $f_P(A)$ for non-positive-definite matrices $A\in\P_n$ (with $\im(A)\subseteq\im(P)$) as 
\begin{equation}\label{eq:fPAsingular}
f_P(A)= \sum_{i=1}^r f(\alpha_i) \Tr(Pv_iv_i^*),
\end{equation}
where we assume that $A\in\P_n$ has rank $r$ with nonzero eigenvalues $\alpha_1,\dots,\alpha_r>0$ and $\alpha_{r+1}=\cdots=\alpha_n=0$ (see, e.g., equation (2.2) in \cite{Rastegin2011}). 

\subsection{Directional derivatives of matrix trace functions}

Let $g:\calA\rightarrow\RR$ be a real-valued function on some subset $\calA\subset\H_n$. For any $A\in\calA$ and any matrix $B\in\H_n$, the (\emph{one-sided}) \emph{directional derivative} of $g$ at $A$ in the direction $B$ is defined as
\[
 dg(A;B) =\lim_{t\rightarrow0^+} \frac{g(A+tB)-g(A)}{t}.
\]
Here we are interested in computing the directional derivatives of functions of the form $f_P(A) = \Tr(Pf(A))$. If $f$ is differentiable and $A$ is positive definite, then these directional derivatives certainly exist, since the function $f$ when extended to positive definite matrices is differentiable as a function of matrices. However, if $A$ is not necessarily positive definite, these directional derivatives may still be computed. Knowing the derivatives is important for determining optimality conditions for certain types of optimization problems that arise in quantum information \cite{Girard2014}. The directional derivatives are presented in Theorem \ref{thm:difffPA1} and make use of the following notation.

Let $f:(0,\infty)\rightarrow\RR$ be a differentiable function. The \emph{first order divided differences} of $f$ defined as
 \[
  f^{[1]}(x,y) = \left\{\begin{array}{ll}
                         f'(x) & \text{if }x=y\\
                         \displaystyle\frac{f(x)-f(y)}{x-y} & \text{if } x\neq y
                        \end{array}
\right.
 \]
 for all $x,y\in(0,\infty)$. For any positive sedefinite $n\times n$ matrix $A\in\P_n$, we may define a linear mapping $\Phi_{f,A}:\H_n\rightarrow\H_n$ of $n\times n$ matrices as follows. If $A=\diag(\alpha_1,\dots,\alpha_n)$ is diagonal,
 we can write $A$ as
\[
 A=\sum_{i=1}^n \alpha_i \, e_ie_i^*,
\]
where $e_1,\dots,e_n\in\CC^n$ are the standard orthonormal basis vectors of $\CC^n$ such that the entries of any other matrix $B\in\H_n$ are given by $B_{ij}=\langle e_i,Be_j\rangle$. The \emph{matrix of divided differences} of $A$ (restricted to the nonzero eigenvalues of $A$) as the matrix $D_{f,A}\in\H_n$ whose entries are given by
\begin{equation}
 (D_{f,A})_{ij} = \left\{\begin{array}{ll}
                         f^{[1]}(\alpha_i,\alpha_j) & \text{if }\alpha_i>0\text{ and }\alpha_j>0\\
                        0 & \text{if }\alpha_i=0\text{ or }\alpha_j=0,
                        \end{array}
\right.
\end{equation}
and for all $B\in\H_n$ define $\Phi_{f,A}(B)$ as
\begin{equation}\label{eq:PhifABdiag}
 \Phi_{f,A}(B) = D_{f,A}\odot B
\end{equation}
where $X\odot Y$ denotes the entrywise product of matrices $X,Y\in\H_n$ with matrix elements $(X\odot Y)_{ij}=X_{ij}Y_{ij}$ for all $i,j\in\{1,\dots,n\}$. If $A$ is not diagonal, there exists an $n\times n$ unitary matrix $U$ such that $UAU^*$ is diagonal, and one defines
\begin{equation}\label{eq:PhifAB}
 \Phi_{f,A}(B) = U^*\bigl(D_{f,UAU^*}\odot (UBU^*)\bigr)
\end{equation}
 for all $B\in\H_n$, and this is independent of the choice of diagonalizing unitary $U$. We may now state the main theorem of this work. 

\begin{theorem}\label{thm:difffPA1}
 Let $f:(0,\infty)\rightarrow\RR$ be a differentiable function satisfying $\lim_{t\rightarrow0^+}tf(t)=0$, let $P\in\P_n$ be a positive semidefinite $n\times n$ matrix, and consider the function $f_P:\P_n^+\rightarrow\RR$ as defined above. Let $A\in\P_n$ be a positive matrix satisfying $\im(P)\subseteq\im(A)$ such that we may define $f_P(A)$ as in \eqref{eq:fPAsingular}, and let $B\in\H_n$. Suppose there exists $\varepsilon>0$ such that $A+tB\in\P_n$ holds for all $t\in[0,\varepsilon)$. The directional derivative of $f$ at $A$ in the direction $B$ exists and can be computed by
 \begin{align}
  df_P(A;B) &= \lim_{t\rightarrow0^+} \frac{f_P(A+tX)-f_P(A)}{t}\nonumber\\
  & = \Tr(P\,\Phi_{f,A}(B)),\label{eq:dfPAB3}
 \end{align}
where $\Phi_{f,A}:\H_n\rightarrow\H_n$ is the linear mapping defined above in \eqref{eq:PhifAB}.
\end{theorem}

In the case when $A$ is positive \emph{definite}, we remark that the directional derivative in \eqref{eq:dfPAB3} coincides with well known results in \cite[Theorem V.3.3]{Bhatia1997} and \cite[Theorem 3.25]{Hiai2014}. For non-positive-definite $A\in\P_n$, these derivatives were provided in \cite{Girard2014}, but no proof of the existence of the directional derivatives were provided there. We note that the function $f$  in Theorem \ref{thm:difffPA1} must satisfy the condition that
\begin{equation}\label{eq:limttftzero}
 \lim_{t\rightarrow0^+} tf(t)=0
\end{equation}
in order for the derivatives to be computed in this manner. 

The remainder of this note is dedicated to the proof of Theorem \ref{thm:difffPA1} (which will be proved using matrix perturbation methods) and to provide some examples. In particular, we note that the condition in the theorem is met for the function $f(t)=\log(t)$ and the functions $f(t)=t^p$ for all real values $p>-1$, as these functions satisfy \eqref{eq:limttftzero}. The derivatives of the function $A\mapsto\Tr(P\log(A))$ at non-positive-definite matrices $A$ were studied in \cite{Friedland2011}. For the function $f(t)=t^{-1}$, the directional derivatives are no longer able to be computed in this manner, as this function does not satisfy the condition in \eqref{eq:limttftzero}, however the expression in \eqref{eq:dfPAB3} still provides a lower bound for the directional derivative. 

The remainder of the note is organized as follows. Section \ref{sec:back} introduces the notation that will be used in this note, recalls some basic notions of differentiation of matrix functions, and presents some facts from perturbation theory for Hermitian matrices. The proof of Theorem \ref{thm:difffPA1} is presented in Section \ref{sec:main}. In Section \ref{sec:alternate}, we consider the functions $f(t)=\log(t)$ and $f(t)=t^p$ for $p\in(-1,1)$ as examples, and provide alternate proofs of the directional derivatives for these functions by the method of integral representations rather than matrix perturbation methods. Finally, in Section \ref{sec:lower}, we show that this method finds a lower bound to the directional derivatives for the choice of function $f(t)=t^{-1}$.

\section{Background }\label{sec:back}

Notions of differentiability of matrix functions are recalled in Section \ref{sec:matrixdif}. Some results on spectral perturbation theory are reviewed in \ref{sec:matperturb}.

\subsection{Derivatives of matrix functions}
\label{sec:matrixdif}

We refer to \cite{Bhatia1997} for more details. Let $\calA\subseteq\H_n $ be a subset of the $n\times n$ Hermitian matrices and let $f:\calA\rightarrow\H_n $ be a function of matrices. The function $f$ is said to be (\emph{Fr\'echet}) \emph{differentiable} at a matrix $A\in\calA$ if there exists a linear mapping $\Phi:\H_n \rightarrow\H_n $ of matrices satisfying
\[
 \lim_{H\rightarrow 0} \frac{\lVert f(A+tH)-f(A) - \Phi(H)\rVert}{\lVert H\rVert} = 0,
\]
where $\lVert\cdot\rVert$ denotes the spectral norm on the space of matrices $n\times n$. If such a mapping exists, it is called the (\emph{Fr\'echet}) \emph{derivative} of $f$ at $A$ and is denoted by $\Phi=Df(A)$. In cases where the function is not differentiable at a point, it may still possess directional derivatives. 

Let $f:(0,\rightarrow\RR)$ be a differentiable functions. If $A\in\P_n^+$ is a positive definite matrix, then $f$ is differentiable (as a function of matrices) at $A$ with Fr\'echet derivative 
\begin{equation}\label{eq:DfAB}
 Df(A)(B) = \Phi_{f,A}(B),
\end{equation}
where $\Phi_{f,A}:\H_n\rightarrow\H_n$ is the linear mapping defined in \eqref{eq:PhifAB}. Moreover, for any positive semidefinite matrix $P\in\P_n$, the directional derivatives of the function $f_P:\P_n^+\rightarrow\RR$ (as defined in \eqref{fPAdef}) at any positive definite matrix $A\in\P_n^+$ are given by
\[
 df_P(A;B) = \Tr(P\,Df(A)(B))
\]
for all $B\in\H_n$.

\subsection{Spectral perturbation theory for Hermitian matrices}
\label{sec:matperturb}

Consider now families of Hermitian matrices of the form $A+tB$ for some choice of Hermitian matrices $A,B\in\H_n $ and variable $t\in\RR$. It is a remarkable fact from perturbation theory of linear opeators (see, e.g., \cite[II.6.2]{Kato1980}) that there exists a spectral decomposition of $A+tB$ that behaves \emph{analytically} in the variable~$t$. That is, there exist analytic functions $\lambda_1,\dots,\lambda_n:\RR\rightarrow\RR$ for the eigenvalues of $A+tB$ and analytic vector-valued functions $u_1,\dots,u_n:\RR\rightarrow\CC^n$ such that $A+tB$ may be expressed as
\begin{equation} \label{eq:AtBspectral}
 A+tB = \sum_{i=1}^n \lambda_i(t)\, u_i(t)u_i(t)^*,
\end{equation}
for all $t\in\RR$. As this is a spectral decomposition of $A+tB$, one has that $(A+tB)u_i(t)=\lambda_i(t)u_i(t)$ and that 
\[
 \langle u_i(t),u_j(t)\rangle = \left\{\begin{array}{ll}
                                            1 & \text{ if }i=j\\
                                            0 & \text{ if }i\neq j
                                           \end{array}
\right.
\]
holds for all $t\in\RR$. Suppose that the eigenvalues $\alpha_1,\dots,\alpha_n$ and the eigenvectors $v_1,\dots,v_n$ of $A$ are such that $\alpha_i=\lambda_i(0)$ and $v_i=u_i(0)$ for each $i\in\{1,\dots,n\}$. The first-order derivatives $\lambda_i'(0)$ and $u_i'(0)$ can be computed from $B$ and the spectral decomposition of $A$, as the following propostion shows.

\begin{proposition}\label{prop:AtBeig}
Suppose $A\in\H_n $ and $B\in\H_n $ are Hermitian matrices and let $\lambda_1(t),\dots,\lambda_n(t)$ and $u_1(t),\dots,u_n(t)$ denote the eigenvalues and coresponding eigenvectors (which are analytic as functions of $t$) of the matrix $A+tB$ comprising the spectral decomposition in \eqref{eq:AtBspectral}. The following statements hold. 

\begin{enumerate}[(i)]

\item For all $i\in\{1,\dots,n\}$, it holds that $\lambda_i'(0) = \langle v_i,B v_i\rangle$.
    
\item For all $i$ and $j$ with $i\neq j$, it holds that $\bigl(\alpha_i-\alpha_j\bigr) \langle v_i,u_j'(0)\rangle =  \langle v_i,Bv_j\rangle$.

\item For all $i$ and $j$, it holds that $\langle u_j'(0),v_j\rangle + \langle v_j,u_j'(0)\rangle=0$.
\end{enumerate}
Here, $\alpha_1,\dots,\alpha_n$ are the eigenvalues and $v_1,\dots,v_n$ are the eigenvectors of  $A$ such that $\alpha_i=\lambda_i(0)$ and $v_i=u_i(0)$ for each $i\in\{1,\dots,n\}$
\end{proposition}
\begin{proof}
 For each index $j$, note that the expression $(A+tB)u_j(t) - \lambda_j(t)u_j(t) =0$ is constant with respect to $t$. Differentiating this expression at $t=0$ yields
\[
 0 = \frac{d}{dt}\bigl((A+tB)u_j(t) - \lambda_j(t)u_j(t)  \bigr)\Bigl|_{t=0} = Bv_j + Au_j'(0) - \lambda_j'(0)v_j - \alpha_ju_j'(0).
\]
Taking the inner product of this expression with $v_i$, one finds that
\begin{align*}
0 & =\bigl\langle v_i, \, \bigl(Bv_j + Au_j'(0) - \lambda_j'(0)v_j - \alpha_ju_j'(0)\bigr)\bigr\rangle\\ 
  & = \langle v_i, Bv_j\rangle + \alpha_i \langle v_i,u_j'(0)\rangle  - \alpha_i \langle v_i,u_j'(0)\rangle- \lambda_i'(0)\langle v_i,v_j\rangle \\
  & = \langle v_i, Bv_j\rangle + \bigl(\alpha_i-\alpha_j\bigr) \langle v_i,u_j'(0)\rangle - \lambda_i'(0)\langle v_i,v_j\rangle .
\end{align*}
Taking $i=j$ yields property (i) while taking $i\neq j$ yields property (ii). To prove (iii), note that $\langle u_i(t),u_j(t)\rangle$ is constant for all $i$ and $j$. Taking the derivative yields
\[
 0 = \frac{d}{dt}\langle u_i(t),u_j(t)\rangle \Bigl|_{t=0}  = \langle u_i'(0),u_j(0)\rangle + \langle u_i(0),u_j'(0)\rangle =\langle u_i'(0),v_j\rangle + \langle v_i,u_j'(0)\rangle,
\]
as desired.
\end{proof}

\section{Proof of Theorem \ref{thm:difffPA1}}
\label{sec:conddfPA} \label{sec:main}

\begin{proof}[Proof \textup{(}of Theorem \ref{thm:difffPA1}\textup{)}]
Let $\lambda_1,\dots,\lambda_n:\RR\rightarrow\RR$ and $u_1,\dots,u_n:\RR\rightarrow\CC^n$ be the analytic functions denoting the eigenvalues and corresponding orthonormal eigenvectors of $A+tB$, and let $\alpha_1,\dots,\alpha_n$ be the eigenvalues and $v_1,\dots,v_n$ the corresponding orthonormal eigenvectors of $A$ comprising the spectral decomposition
 \[
  A=\sum_{i=1}^n \alpha_i\, v_iv_i^*,
 \]
 such that $\lambda_i(0)=\alpha_i$ and $u_i(0)=v_i$ for all $i\in\{1,\dots,n\}$. We may assume that $\alpha_1,\dots,\alpha_r>0$ are the nonzero eigenvalues of $A$ and that $\alpha_{r+1}=\dots=\alpha_n=0$. Define the value of $f_P(A)$ as 
 \[
  f_P(A) = \sum_{i=1}^r f(\alpha_i)\langle v_i,Pv_i\rangle,
 \]
 where one sums only over the nonzero eigenvalues of $A$. Note from Proposition \ref{prop:AtBeig} that $\langle v_i, B v_i\rangle =\lambda_i'(0)$, and moreover that $\lambda_i'(0)\geq0$ must hold by assumption for all $i\in\{r+1,\dots,n\}$ since $A+tB$ is assumed to be positive semidefinite for all $t\in[0,\varepsilon)$. Furthermore, it may assumed without loss of generality that $\langle v_i, B v_i\rangle >0$ for all $i\in\{r+1,\dots,n\}$. Indeed, if it holds that $\langle v_i, B v_i\rangle =0$ for some $i\in\{r+1,\dots,n\}$ then $\langle v_i, (A+tB) v_i\rangle =0$ holds for all $t$, and one may restrict to the problem to the subspace perpendicular to $v_i$. One therefore has that
\begin{align}
 df_P(A;B)&=\lim_{t\rightarrow 0^+} \frac{f_P(A+tB) - f_P(A)}{t}\nonumber\\
 & = \lim_{t\rightarrow 0^+} \frac{\Tr(P\, f(A+tB)) - f_P(A)}{t}\nonumber\\   
 & = \lim_{t\rightarrow 0^+}\sum_{i=1}^r \frac{
f(\lambda_i(t))\langle u_i(t),Pu_i(t)\rangle -f(\alpha_i) \langle v_i,Pv_i\rangle
 }{t} + 
\lim_{t\rightarrow 0^+}\sum_{i=r+1}^n \frac{f(\lambda_i(t))\langle u_i(t),Pu_i(t)\rangle}{t}\nonumber\\
& = \sum_{i=1}^r f'(\alpha_i)\lambda_i'(0) \langle u_i(0),Pu_i(0)\rangle  + \sum_{i=1}^r
f(\alpha_i) 
\left( \langle u_i'(0),Pu_i(0)\rangle +  \langle u_i(0),Pu_i'(0)\rangle\right)\label{eq:dfPAB1}\\
&\hspace{3.5in}
+\sum_{i=r+1}^n \lim_{t\rightarrow 0^+}\frac{f(\lambda_i(t))h_i(t)}{t}, \nonumber
\end{align}
where we define the functions $h_i(t)=\langle u_i(t),Pu_i(t)\rangle$ for each $i\in\{r+1,\dots,n\}$. Note that each $h_i$ is analytic with $h_i(0)=h_i'(0)=0$, since $Pv_i=0$ holds for all $i\in\{r+1,\dots,n\}$ by the assumption that $\im(P)\subset\im(A)$. The second sum in \eqref{eq:dfPAB1} reduces to
\begin{align*}
 &\sum_{i=1}^r f(\alpha_i)  \Bigl(\langle u_i'(0),Pu_i(0)\rangle +  \langle u_i(0),Pu_i'(0)\rangle\Bigr) \\
&\hspace{1cm}= \sum_{i,j=1}^r f(\alpha_i)  \Bigl( \langle u_i'(0),v_j\rangle\langle v_j,Pv_i\rangle +  \langle v_i,Pv_j\rangle\langle v_j,u_i'(0)\rangle\Bigr)\\
&\hspace{1cm}= \sum_{i,j=1}^r   \langle v_j,Pv_i\rangle  \left(f(\alpha_i)\langle u_i'(0),v_j\rangle + f(\alpha_j)\langle v_i,u_j'(0)\rangle\right)\\
& \hspace{1cm}=\sum_{\substack{i,j=1\\\alpha_i= \alpha_j}}^r   f(\alpha_i)\langle v_j,Pv_i\rangle  \underbrace{\left(\langle u_i'(0),v_j\rangle + \langle v_i,u_j'(0)\rangle\right)}_{=0} + \sum_{\substack{i,j=1\\\alpha_i\neq \alpha_j}}^r\langle v_j,Pv_i\rangle  \left(f(\alpha_i)\langle u_i'(0),v_j\rangle + f(\alpha_j)\langle v_i,u_j'(0)\rangle\right)
\end{align*}
where the term $\langle u_i'(0),v_j\rangle + \langle v_i,u_j'(0)\rangle$ in the last line above vanishes by statement (iii) in Proposition \ref{prop:AtBeig}. Thus the second sum in \eqref{eq:dfPAB1} further reduces to 
\begin{align*}
&  \sum_{\substack{i,j=1\\\alpha_i\neq \alpha_j}}^r\langle v_j,Pv_i\rangle  \left(f(\alpha_i)\frac{\langle v_i,Bv_j\rangle}{\alpha_j-\alpha_i} + f(\alpha_j)\frac{\langle v_i,Bv_j\rangle}{\alpha_i-\alpha_j}\right)\\
& \hspace{2cm}= \sum_{\substack{i,j=1\\\alpha_i\neq \alpha_j}}^r\langle v_j,Pv_i\rangle \langle v_i,Bv_j\rangle  \frac{f(\alpha_i)-f(\alpha_j)}{\alpha_i-\alpha_j} \\
&\hspace{2cm}=\sum_{\substack{i,j=1\\\alpha_i\neq \alpha_j}}^r\langle v_j,Pv_i\rangle \langle v_i,Bv_j\rangle  f^{[1]}(\alpha_i,\alpha_j).
\end{align*}
Noting from statement (ii) of Proposition \ref{prop:AtBeig} that $\langle v_i,Bv_j\rangle=0$ for all pairs of indices $i\neq j$ with $\alpha_i=\alpha_j$, the first two sums in \eqref{eq:dfPAB1} reduce to
\begin{align}
 \sum_{i=1}^r &f'(\alpha_i)\lambda_i'(0) \langle u_i(0),Pu_i(0)\rangle  + \sum_{i=1}^r f(\alpha_i) \left( \langle u_i'(0),Pu_i(0)\rangle +  \langle u_i(0),Pu_i'(0)\rangle\right) \nonumber\\
 &\hspace{2cm}= \sum_{i=1}^r \langle v_j,Pv_i\rangle \langle v_i,Bv_i\rangle f'(\alpha_i) +  \sum_{\substack{i,j=1\\\alpha_i\neq\alpha_j}}^n \langle v_j,Pv_i\rangle \langle v_i,Bv_j\rangle f^{[1]}(\alpha_i,\alpha_j)\nonumber\\
&\hspace{2cm}= \sum_{i,j=1}^n \langle v_j,Pv_i\rangle \langle v_i,Bv_j\rangle f^{[1]}(\alpha_i,\alpha_j) \nonumber\\
&\hspace{2cm}= \Tr(P\,\Phi_{f,A}(B))\label{eq:final1}.
\end{align}
Finally, as $\lambda_i(0)=\alpha_i=0$ and $\lambda_i'(0)=\langle v_i,Bv_i\rangle>0$ for all $i\in\{r+1,\dots,n\}$, it holds that
\begin{equation}\label{eq:limit1}
 \lim_{t\rightarrow0^+}\frac{f(\lambda_i(t))h_i(t)}{t} = \lim_{t\rightarrow0^+}tf(\lambda_i(t))\frac{h_i(t)}{t^2} = \underbrace{\lim_{t\rightarrow0^+}tf(\lambda_i(t))}_{=0}\frac{h_i''(0)}{2} =0,
\end{equation}
for all $i\in\{r+1,\dots,n\}$, where the limit vanishes from the fact that $\lambda_{i}(0)=0$ and $\lambda'_i(0)>0$, and by the assumption that $\lim_{t\rightarrow0^+}tf(t)=0$. Plugging the results of \eqref{eq:final1} and \eqref{eq:limit1} into \eqref{eq:dfPAB1} yields $df_P(A;B) = \Tr(P\,\Phi_{f,A}(B))$, as desired.
\end{proof}

\section{Alternative proofs of differentiability via integral representations} \label{sec:alternate}

In this section, alternative proofs for the computations of the directional derivatives of $f_{P}$ are provided in the case when $f(x)=\log (x)$ or $f(x)=x^p$ for some value $p\in(-1,1)$ following the method in \cite[Thm.\, 3]{Vedral1998}. This method makes use of integral representations of these functions, which may be extended to matrices in the usual way. 

\subsection{Directional derivatives of \texorpdfstring{$f_{p,P}$}{fpP}}

Let $p\in(-1,0)\cup(0,1)$ and consider the function $f_p:(0,\infty)\rightarrow\RR$ defined as $f_p(x)=x^p$ for all $x\in(0,\infty)$. The divided differences of this function are given by
\[
 f_p^{[1]}(x,y) = \left\{\begin{array}{ll}
                                p x^{p-1}& x=y\\
                                \displaystyle\frac{x^p-y^p}{x-y} & x\neq y.
                               \end{array}
\right.
\]
For a positive matrix $P\in\P_n $, consider the function $f_{p,P}$ defined on positive matrices as
\[
 f_{p,P}(A) = \left\{ \begin{array}{ll}
                            \Tr(PA^{p}) & \im(P)\subseteq\im(A)\text{ or }p\in(0,1)\\
                            +\infty & \text{else}
                           \end{array}\right.
\]
for all $A\in\P_n $. Note that $f_{p,P}$ is differentiable at all positive definite matrices $A$, as the function $f_p$ is differentiable. As indicated by Theorem \ref{thm:difffPA1}, the directional derivatives of $f_{p,P}$ at a positive semidefinite matrix $A$ can be computed as
\[
 df_{p,P}(A;B)  =\Tr(P \,\Phi_{f_p,A}(B))
\]
as long as $\im(P)\subseteq\im(A)$, where $B\in\H_n$ is any Hermitian matrix such that $A+tB$ is positive for all $t>0$ small enough. Here we show how to directly compute these directional derivatives using a method of integral representations for $f_p$.

The calculation is split into the cases $p\in(-1,0)$ and $p\in(0,1)$, which are considered in Sections \ref{sec:casealpha12} and \ref{sec:casealpha01} respectively. The following integral representations will be used.\footnote{c.f.\ \cite[Lemma2.8]{Carlen2010}} For all $x\in(0,\infty)$ one has
\begin{align}
 x^{p} &= \frac{-\sin(p\pi)}{\pi}\int_{0}^\infty \frac{s^{p}}{x+s}ds \label{eq:xp_intrep_1} &&\text{for all }p\in(-1,0)\\
 \text{and}\qquad x^{p} & = \frac{\sin(p\pi)}{\pi}\int_{0}^\infty s^{p}\left(\frac{1}{s}-\frac{1}{x+s}\right)ds &&\text{for all }p\in(0,1).\label{eq:xp_intrep_2}
\end{align}
Furthermore, for all $x,y\in(0,\infty)$ with $x\neq y$, and all $p\in(-1,1)$, one has 
\begin{align*}
 px^{p-1} &= \frac{\sin(p\pi)}{\pi}\int_{0}^\infty \frac{s^p}{(x+s)^2}ds\\
\text{and}\qquad \frac{x^p-y^p}{x-y} & = \frac{\sin(p\pi)}{\pi}\int_{0}^\infty \frac{s^p}{(x+s)(y+s)}ds.
\end{align*}
In particular, for all $p\in(-1,1)$, the divided differences of the function $f_p:(0,\infty)\rightarrow\RR$ defined by $f_{p}=x^{p}$ can be given by
\begin{equation}\label{eq:intrep_falphadivideddiffs}
 f_p^{[1]}(x,y) = \frac{\sin(p\pi)}{\pi}\int_{0}^\infty \frac{s^{p}}{(x+s)(y+s)}ds
\end{equation}
for all $x,y\in(0,\infty)$. For the function $g:(0,\infty)\rightarrow\RR$ defined as $g(x)=x^{-1}$ for all $x\in(0,\infty)$, note that the divided differences can expressed compactly as
\begin{equation}\label{eq:g1xy}
 g^{[1]}(x,y) = -\frac{1}{xy}
 \end{equation}
for all $x,y\in(0,\infty)$, since $g'(x)=-1/x^2$ and $g^{[1]}(x,y) = (1/x-1/y)/(x-y) = -1/xy$ for all $x,y\in(0,\infty)$ with $x\neq y$. For any positive definite matrix $A\in\P_n^+ $ and any other matrix $B\in\H_n$, one has that
\[
 \lim_{t\rightarrow0}\frac{(A+tB)^{-1}-A^{-1}}{t} = Dg(A)(B),
\]
where $Dg(A):\H_n\rightarrow\H_n$ is the linear Fr\'echet differential operator (as defined in \eqref{eq:DfAB}) for $g(x)=x^{-1}$. In the case when $A=\diag(\alpha_1,\dots\alpha_n)$ is diagonal and positive definite, the $ij$-entry of the matrix $Dg(A)(B)$ are computed as
\[
 \bigl(Dg(A)(B)\bigr)_{ij} = g^{[1]}(\alpha_i,\alpha_j) B_{ij} = -\frac{B_{ij}}{\alpha_i\alpha_j}
\]
for any $B\in\H_n$.

\subsubsection{The case \texorpdfstring{$p\in(-1,0)$}{p(in(-1,0)}} \label{sec:casealpha12}

First consider the case when $p\in(-1,0)$. Let $A\in\P_n$ be a positive matrix satisfying $\im(P)\subseteq\im(A)$, which we may suppose without loss of generality is diagonal with $A=\diag(\alpha_1,\dots,\alpha_n)$.  We may assume that $\alpha_1,\dots,\alpha_r>0$ are the nonzero eigenvalues and that $\alpha_{r+1}=\cdots=\alpha_n=0$. Let $B\in\H_n$ be an $n\times n$ Hermitian matrix and suppose there exists a positive value $\varepsilon>0$ such that $A+tB\in\P_n $ for all $t\in[0,\varepsilon)$. One may compute $f_{\alpha,P}(A+tB)$ for any $t\in[0,\varepsilon)$ using the integral representation in \eqref{eq:xp_intrep_1} as
\begin{align*}
 f_{\alpha,P}(A+tB) &=  \Tr(P (A+tB)^{p})\\
 & =\frac{-\sin(p\pi)}{\pi}\int_0^\infty\Tr\Bigl(P  (A+tB + s\I)^{-1}\Bigr)s^{p}\, ds,
\end{align*}
whre $\I$ denotes the $n\times n$ identity matrix. This holds even when $t=0$. The directional derivative $df_{p,P}(A;B)$ can be computed as
\begin{align*}
 df_{p,P}(A;B) &= \lim_{t\rightarrow 0^+} \frac{f_{\alpha,P}(A+tB)-f_{\alpha,P}(A)}{t}\\
  & = \lim_{t\rightarrow 0^+} \frac{\frac{-\sin(p\pi)}{\pi}\displaystyle\int_0^\infty\Tr\Bigl(P  \left((A+tB + s\I)^{-1} -(A + s\I)^{-1} \right)\Bigr)s^{p}\, ds}{t}\\
  & = \frac{-\sin(p\pi)}{\pi}\int_0^\infty\Tr\Biggl(P  \lim_{t\rightarrow 0^+} \frac{(A+tB + s\I)^{-1} -(A + s\I)^{-1}}{t}\Biggr)s^{p}\, ds\\
  & = \frac{-\sin(p\pi)}{\pi}\int_0^\infty  \Tr\left(P \,Dg(A+s\I)(B)\right)s^{p}\, ds
\end{align*}
where $g:(0,\infty)\rightarrow\RR$ is the function (defined earlier) $g(x)=x^{-1}$ for all $x\in(0,\infty)$. Note that $A+s\I$ is positive definite and diagonal for all $s\in(0,\infty)$ with eigenvalues $\alpha_i+s$. Extending $g$ to all positive definite matrices, one sees that $g$ is Fr\'echet differentiable at the positive definite matrix $A+s\I$ for all $s>0$ where the matrix entries of the derivative $Dg(A+s\I)(B)$ are given by
\[
 \bigl(Dg(A+s\I)(B)\bigr)_{i,j} =  g^{[1]}(\alpha_i+s,\alpha_j+s)\langle e_i,B e_j\rangle = -\frac{\langle e_i,B e_j\rangle}{(\alpha_i+s)(\alpha_j+s)}
\]
for all $i,j\in\{1,\dots,n\}$, and the divided differences are computed as in \eqref{eq:g1xy}. As it has been assumed that $\im(P)\subseteq\im(A)$, it holds that $P_{ij}=\Tr(P\,e_ie_j^*)=0$ whenever $i\in\{r+1,\dots,n\}$ or $j\in\{r+1,\dots,n\}$. It follows that
\[
 \Tr(P\, Dg(A+s\I)(B)) = -\sum_{i,j=1}^r \frac{\langle e_i,B e_j\rangle\langle e_j,P e_i\rangle}{(\alpha_i+s)(\alpha_j+s)}
\]
for all $s>0$, where one notes that the sum above is taken from $1$ to $r$. Making use of the integral representation for the divided differences $f_p^{[1]}(\alpha_i,\alpha_j)$ in \eqref{eq:intrep_falphadivideddiffs}, it follows that
\begin{align*}
 df_{p,P}(A;B)  &= \sum_{i,j=1}^r\langle e_i,B e_j\rangle\langle e_j,P e_i\rangle\frac{\sin(p\pi)}{\pi}\int_0^\infty \frac{s^{p}}{(\alpha_i+s)(\alpha_j+s)}\, ds\\
  & =\sum_{i,j=1}^re_a^*\langle e_i,B e_j\rangle\,\langle e_j,P e_i\rangle\, f_p^{[1]}(\alpha_i,\alpha_j)\\
  & = \Tr(P \,\Phi_{f_p,A}(B)),
\end{align*}
where $\Phi_{f_p,A}$ is the linear mapping defined earlier.

\subsubsection{The case \texorpdfstring{$p\in(0,1)$}{pin(0,1)} }\label{sec:casealpha01}

Now let $p\in(0,1)$. One may compute $f_{p,P}(A+tB)$ using integral representation in \eqref{eq:xp_intrep_2} as 
\begin{align*}
 f_{p,P}(A+tB) &=  \Tr(P (A+tB)^{p})\\
 & =\frac{\sin(p\pi)}{\pi}\int_0^\infty\Tr\Bigl(P \left( s^{-1}\I-(A+tB + s\I)^{-1}\right)\Bigr)s^{p}\, ds,
\end{align*}
for all $t\in[0,\varepsilon)$. The directional derivative $df_{p,P}(A,B)$ can be computed by
\begin{align*}
 df_{p,P}(A;B) &= \lim_{t\rightarrow 0^+} \frac{f_{p,P}(A+tB)-f_{p,P}(A)}{t}\\
  & = \lim_{t\rightarrow 0^+} \frac{\frac{\sin(p\pi)}{\pi}\displaystyle\int_0^\infty\Tr\Bigl(P  \left((A + s\I)^{-1} -(A +tB+ s\I)^{-1} \right)\Bigr)s^{p}\, ds}{t}\\
  & = \frac{\sin(p\pi)}{\pi}\int_{0}^{\infty} \Tr\Biggl(P\lim_{t\rightarrow0^+}\frac{(A+s\I)^{-1}-(A+tB+s\I)^{-1}}{t}\Biggr)s^{p}\, ds\\
  & = \frac{-\sin(p\pi)}{\pi}\int_0^\infty  \Tr\Bigl(P\,Dg(A+s\I)(B)\Bigr)s^{p}\, ds\\
  &= \sum_{i,j=1}^r\langle e_i,B e_j\rangle\,\langle e_j,P e_i\rangle\, f_p^{[1]}(\alpha_i,\alpha_j)\\
  & = \Tr(P\,\Phi_{f_p,A}(B)) ,
\end{align*}
using the same arguments as before.

\subsection{Directional derivatives of \texorpdfstring{$\Tr(P\log(A))$}{Tr(PlogA)}}

The same methods can be used to compute the derivatives of $\Tr(P\log(A))$. One may use the integral representation of logarithm function, which holds for all $x\in(0,\infty)$:
\[
 \log(x)=\int_0^\infty
\left(\frac{1}{s}-\frac{1}{x+s}\right)ds.
\]
Let $f:(0,\infty)\rightarrow\RR$ be the function defined as $f(x)=\log x$ for all $x\in(0,\infty)$. For a positive matrix $P\in\P_n$, define the function $f_P$ as
\[
 f_{P}(A)=\Tr(P\log A) = \int_0^\infty \Tr\left(P\left(\frac{1}{s}\I - (A+s\I)^{-1}\right)\right)ds
\]
for all $A\in\P_n$. Let $A\in\P_n$ be a matrix satisfying $\im(P)\subseteq\im(A)$. One may suppose without loss of generality that $A=\diag(\alpha_1,\dots,\alpha_n)$ is diagonal. Let $B\in\H_n$ and suppose there is a value $\varepsilon>0$ such that $A+tB\in\P_n$ holds for all $t\in[0,\varepsilon)$.  
Then
\begin{align*}
 df_{P}(A;B) &= \lim_{t\rightarrow0^+}\frac{f_P(A+tB)-f_{P}(A)}{t}\\
   &= \lim_{t\rightarrow0^+}\frac{\Tr(P\log(A+tB))-\Tr(P\log(A))}{t}\\
    & = \int_{0}^\infty \Tr\left(P\lim_{t\rightarrow0^+}\frac{(A+s\I)^{-1}-(A+tB+s\I)^{-1}}{t}\right)ds\\
    & = -\int_{0}^\infty \Tr\left(P\, Dg(A+s\I)(B)\right)ds\\
    &=\sum_{i,j=1}^r\langle e_i,B e_j\rangle\,\langle e_j,P e_i\rangle\, f^{[1]}(\alpha_i,\alpha_j)\\
    & =\Tr(P\,\Phi_{f,A}(B)),
\end{align*}
where the steps are analogous to those in Section \ref{sec:casealpha01}. This generalizes the method in \cite[Theorem 3]{Vedral1998}.

Note that the divided differences of the function $f(t)=\log(t)$ are given by
\[
 f^{[1]}(x,y) = \left\{ \begin{array}{ll}
                         \displaystyle\frac{1}{x} & x=y\\
                         \displaystyle \frac{\log(x)-\log(y)}{x-y} & x\neq y
                        \end{array}
 \right.
\]
for all $x,y\in(0,\infty)$.

\section{Lower bound for derivative of \texorpdfstring{$A\mapsto\Tr(PA^{-1})$}{AmapstoTr(PA(-1))}}
\label{sec:lower}

We now consider the function $f:(0,\infty)\rightarrow\RR$ defined by $f(t)=t^{-1}$ and the corresponding trace functional $f_P:\P_n^+\rightarrow\RR$ defined as
\[
 f_P(A)=\Tr(PA^{-1})
\]
for all positive definite matrices $A$. Let $A\in\P_n$ be a positive semidefinite matrix with spectral decomposition
\[
 A=\sum_{i=1}^n\alpha_i\, v_iv_i^*,
\]
where $\alpha_1,\dots,\alpha_r>0$ are the nonzero eigenvalues.
Let $B\in\H_n$ and suppose there exists a positive value $\varepsilon>0$ such that $A+tB\in\P_n$ for all $t\in[0,\varepsilon)$. Let $\lambda_1(t),\dots,\lambda_n(t)$ and $u_1(t),\dots,u_n(t)$ be the analytic eigenvalues and eigenvectors of $A+tB$ such that $\lambda_i(0)=\alpha_i$ and $u_i(0)=v_i$ for all $i\in\{1,\dots,n\}$. As in the proof of Theorem \ref{thm:difffPA1}, for each  $i\in\{r,\dots,n\}$ we define the function $h_i:\RR\rightarrow\RR$  by $h_i(t)=\langle u_i(t),Pu_i(t)\rangle$ such that $h_i(0)=h_i'(0)=0$. Moreover, note that 
\[
 h''_i(0) = \langle u_i'(0),Pu_i'(0)\rangle \geq0,
\]
since $P$ is positive semidefinite. Furthermore, we may assume (as in the proof of Theorem \ref{thm:difffPA1}) that $\lambda_i'(0)>0$ holds for all $i\in\{r+1,\dots,n\}$. Then
\begin{align}
 df_P(A;B)&=\lim_{t\rightarrow 0^+} \frac{f_P(A+tB) - f_P(A)}{t}\nonumber\\
 & = \Tr(P\,\Phi_{f,A}(B))+ \sum_{i=r+1}^n \lim_{t\rightarrow 0^+}\frac{f(\lambda_i(t))h_i(t)}{t} \nonumber\\
 & = \Tr(P\,\Phi_{f,A}(B))+ \sum_{i=r+1}^n \lim_{t\rightarrow 0^+}\frac{h_i(t)}{t\lambda_j(t)},
\end{align}
where, for $i\in\{r+1,\dots,n\}$, the limits in the final line reduce to
\[
 \lim_{t\rightarrow 0^+}\frac{h_i(t)}{t\lambda_i(t)} = \frac{h_i''(0)}{2} \lim_{t\rightarrow 0^+}\frac{t}{\lambda_i(t)} =  \frac{h_i''(0)}{2\lambda_i'(0)} \geq0
\]
since $\lambda_i'(0)=0$ and $h_i''(0)\geq0$. Thus $\Tr(P\,\Phi_{f,A}(B))$ provides the lower bound for the directional derivative,
\begin{equation}\label{eq:dfPABineq}
 df_P(A;B) \geq \Tr(P\,\Phi_{f,A}(B)),
\end{equation}
and this inequality is strict in general unless $\im(B)\subseteq\im(A)$. 

Indeed, to show that the inequality in \eqref{eq:dfPABineq} can be strict, consider the following example. Let $A,P\in\P_2$ and $B\in\H_2$ be the $2\times 2$ matrices
\[
 A=P = \begin{pmatrix}
        1 & 0\\ 0 & 0
       \end{pmatrix}
\qquad\text{and}\qquad
B = \begin{pmatrix}
        0 & 1\\ 1& 1
       \end{pmatrix}.
\]
For the function $f(x)=x^{-1}$, we may define $f_P(A)$ as $f_{P}(A)=1$, and the linear mapping $\Phi_{f,A}:\H_2\rightarrow\H_2$ is given by
\[
 \Phi_{f,A} \left(\begin{pmatrix}
        a & \ast\\ \ast& \ast
       \end{pmatrix}\right) = \begin{pmatrix}
        a & 0\\ 0& 0
       \end{pmatrix}
\]
(i.e., it simply picks out the entry in the upper-left corner and zeros out the other entries). It follows that $\Tr(P\Phi_{f,A}(B))=0$ for these matrices, but that
\[
 (A+tB)^{-1} = \begin{pmatrix}
        1 & t\\ t & t
       \end{pmatrix}^{-1} = \frac{1}{1-t}\begin{pmatrix}
        1 & -1\\ -1 & 1/t
       \end{pmatrix}
\]
such that $\Tr(P(A+tB)^{-1})=1/(1-t)$ for all $t>0$, and thus
\begin{align*}
 df_P(A;B) &= \lim_{t\rightarrow0^+} \frac{f_P(A+tB)-f_P(A)}{t} \\
 &= \lim_{t\rightarrow0^+} \frac{\Tr\left(P(A+tB)^{-1}\right) - 1}{t} = \lim_{t\rightarrow0^+}\frac{1}{1-t} =1.
\end{align*}
Hence $df_P(A;B) > \Tr(P\Phi_{f,A}(B))$ for these matrices.

\bibliographystyle{alpha}
\bibliography{C:/Users/m4girard/Documents/library}

\end{document}